\documentclass[twocolumn]{autart}    

\usepackage{graphicx}          

\usepackage{amsmath,amssymb,amsfonts}
\usepackage{mathtools}
\usepackage{caption}
\usepackage{subcaption}
\usepackage{balance}
\usepackage[dvipsnames]{xcolor}
\usepackage{enumerate}
\usepackage{enumitem}
\usepackage{stmaryrd}
\DeclarePairedDelimiter{\iv}{\llbracket}{\rrbracket}

\usepackage{cite}
\usepackage{url}

\newcommand{\R}{\mathbb{R}}
\newcommand{\I}{\mathbb{I}}
\newcommand{\norm}[1]{\left\lVert#1\right\rVert}
\newcommand{\fe}{\mathsf{f}}
\newcommand{\ve}{\mathsf{v}}
\newcommand{\ee}{\mathsf{e}}
\newcommand{\pe}{\mathsf{p}}
\newcommand{\we}{\mathsf{w}}
\newcommand{\Svec}{(y,\mathbf{u},y^+)}
\newcommand{\RCIvec}{(y,\mathbf{u},y)}
\newcommand{\bfu}{\mathbf{u}}

\newcommand\scalemath[2]{\scalebox{#1}{\mbox{\ensuremath{\displaystyle #2}}}}

\begin{document}

\begin{frontmatter}

\title{Configuration-Constrained Tube MPC for Difference-of-Convex Nonlinear Systems\thanksref{footnoteinfo}} 

\thanks[footnoteinfo]{Support from grants PID2022-141159OB-I00 and PID2022-142946NA-I00 funded by MICIU/AEI/ 10.13039/501100011033 and by ERDF/EU is gratefully acknowledged. F.~Fele also acknowledges support from
grant RYC2021-033960-I funded by MICIU/AEI/ 10.13039/501100011033 and European Union NextGenerationEU/PRTR.\\This paper was not presented at any IFAC meeting. Corresponding author F.~Badalamenti.}

\author[Italy]{Filippo Badalamenti}\ead{filippo.badalamenti@imtlucca.it},    
\author[Sevilla1]{Jose A. Borja-Conde}\ead{jaborja@us.es},               
\author[Sevilla2]{Filiberto Fele}\ead{ffele@us.es},  
\author[Sevilla2]{Teodoro Alamo}\ead{talamo@us.es},               
\author[Sevilla2]{Daniel Limon}\ead{dlm@us.es}               

\address[Italy]{IMT School for Advanced Studies Lucca, Italy}  
\address[Sevilla1]{Universidad Loyola Andalucia, Spain}             
\address[Sevilla2]{Dept.~of Systems Engineering and Automation, University of Seville, Av.~de los Descubrimientos s/n, 41092 Seville, Spain}        

\begin{keyword}                           
Nonlinear predictive control; Robust control; Control of constrained systems; Invariant sets; Convex optimization; Uncertain systems.
\end{keyword}                             

\begin{abstract}
This paper develops a convex tube model predictive control formulation for constrained nonlinear systems. We consider dynamics described by a  discrete-time state-space model with parametric and additive uncertainty that admits a difference-of-convex decomposition. Convex directional bounds of the nonlinear dynamics are combined with configuration-constrained polytopic tubes, whose predefined combinatorial structure yields an affine parameterization of their vertices. The resulting finite-dimensional sufficient conditions certify robust one-step tube propagation under parametric uncertainty and additive disturbances, while allowing the tube geometry and an associated vertex control law to be optimized jointly in a single convex program. An implicit terminal condition guarantees recursive feasibility and convergence of the predicted tube to a target robust control invariant set. Numerical results illustrate the closed-loop properties and the trade-off between geometric flexibility and computational complexity.
\end{abstract}

\end{frontmatter}

\section{Introduction}
\label{sec:introduction}

Tube model predictive control (MPC) is a well-established framework for robust constraint satisfaction in uncertain linear systems~\cite{fb,Kolmanovsky1998}. Rather than predicting a single state trajectory, tube MPC optimizes a sequence of sets enclosing all trajectories generated by the admissible uncertainty realizations. For linear systems and linear difference inclusions, this set-based description leads to tractable formulations with recursive feasibility and stability guarantees~\cite{Langson2004,MAYNE_RakovicRigidTMPC2005219,Rakovic2013_LDI_HTMPC}.

Robust tube propagation is considerably more difficult for nonlinear systems. The image of a convex set under a nonlinear map  cannot generally be characterized through finitely many boundary evaluations, and computing a tight successor set may require solving a nonlinear robust optimization problem~\cite{DIEHL20081279}. Tractable robust nonlinear MPC formulations therefore rely on conservative propagation bounds obtained, for instance, through interval arithmetic~\cite{bravo2005computation}, zonotopes~\cite{bravo2003robust}, linearization or polynomial enclosures~\cite{althoff2008reachability,althoff2013reachability}, Lipschitz and incremental stability arguments~\cite{polver2025robust,Kohler2021}, dynamic tubes~\cite{Lopez2019}, contraction metrics~\cite{Sasfi2023}, or successive convexification~\cite{Lishkova2025}. These approaches offer different trade-offs between modeling assumptions, computational complexity, and conservatism. 
Difference-of-convex (DC) decompositions provide a useful structure for constructing tractable nonlinear propagation bounds. Affine minorants of the associated convex components yield convex outer approximations of the dynamics over the considered domain. This property has been exploited in the computation of robust invariant sets for nonlinear and convex difference-inclusion systems~\cite{FIACCHINI2010,FIACCHINI2012819}, and more recently in robust nonlinear tube MPC~\cite{NLTMPC_9993390,doff2026computationally}. These developments address the convexification of nonlinear propagation, but leave the open problem of how such bounds can be combined with a flexible tube geometry which can be optimized online.

For uncertain linear systems, configuration-constrained tube MPC (CC-TMPC) provides such a parameterization~\cite{CCTMPC_VILLANUEVA2024111543}. This relies on configuration-constrained polytopes, which retain a fixed combinatorial structure as their facet offsets vary, so that their vertices depend affinely on the offset vector. This property allows the tube geometry and an associated vertex control law to be optimized jointly within a convex program. The additional geometric flexibility can reduce conservatism relative to more restrictive tube parameterizations, while its computational cost can be adjusted through restrictions on the template and optimization variables~\cite{badalamenti2025_efficientCCTMPC}. Existing CC-TMPC formulations, however, are limited to uncertain linear dynamics.

This paper extends the configuration-constrained tube framework to uncertain nonlinear systems admitting a DC decomposition. Convex directional bounds are constructed along the template facet normals and evaluated at the vertices of the tube section and the parameter set. Their combination with the affine vertex parameterization yields a finite-dimensional convex certificate for robust one-step propagation under parametric uncertainty and additive disturbances. Based on this certificate, we formulate a nonlinear CC-TMPC controller that jointly optimizes the tube geometry and vertex inputs in a single convex program. An implicit terminal condition is introduced to establish recursive feasibility and convergence of the predicted tube to a target robust control invariant set. The method uses a fixed DC outer approximation and is therefore complementary to approaches based on successive online convexification. Numerical results examine the closed-loop behavior and the trade-off between feasible-region size and computational complexity as the template resolution is varied.

The remainder of the paper is organized as follows. Section~\ref{sec:Sect2} defines the system class, and introduces robust control tubes and configuration-constrained polytopes. Section~\ref{sec:convex_reformulation} derives the convex propagation certificate. Section~\ref{sec:CC-TMPC} presents the resulting MPC formulation and its closed-loop properties. Section~\ref{sec:example} reports the numerical results, followed by conclusions in Section~\ref{sec:conclusions}.

\textit{Notation}: $\I_n$ denotes the $n\times n$ identity matrix. For column vectors $a$ and $b$, $(a,b)$ denotes the stacked vector $[a^\top\ b^\top]^\top$. Given matrices $Q_1,\ldots,Q_r$, $\operatorname{blkd}(Q_1,\ldots,Q_r)$ denotes the corresponding block-diagonal matrix. For any set $X$, $\operatorname{conv}(X)$ and $\operatorname{ri}(X)$ denote its convex hull and relative interior, respectively. For matrices $A$ and $B$, $A\otimes B$ denotes their Kronecker product. The notation $A\succ B$ and $A\succeq B$ means that $A-B$ is positive definite and positive semidefinite, respectively. For $Q=Q^\top\succ0$, we define $\|z\|_Q^2=z^\top Qz$. Vector inequalities are understood componentwise, and $\mathbf 1_n$ denotes the vector of ones in $\mathbb R^n$. The unit simplex is $\Delta^{n-1}\coloneqq\{\lambda\in\mathbb R_+^n\mid\mathbf 1_n^\top\lambda=1\}$. Finally, $\iv{c,d}$ denotes the integer set $\{i\in\mathbb Z\mid c\le i\le d\}$.

\section{Problem Formulation}
\label{sec:Sect2}
This section introduces the uncertain system model and the optimization framework which underpins the proposed robust MPC scheme. We first define the class of nonlinear dynamics of interest, and the DC assumption used to obtain their convex outer approximation. We then formalize robust control tubes (RCTs), and show how configuration‑constrained polytopes provide a tractable vertex‑based reformulation.

\subsection{Uncertain Nonlinear Systems}
\label{subsec:DC}

We consider the discrete-time uncertain nonlinear system
\begin{equation} 
\label{eq:sys} x_{t+1} = f(x_t,u_t,\theta_t)+w_t, \qquad \theta_t\in\Theta,\quad w_t\in\mathbb W, 
\end{equation}
where $x_t \in \mathbb{X} \subset \mathbb{R}^{n_x}$ is the state, $u_t \in \mathbb{U} \subset \mathbb{R}^{n_u}$ is the input, $\Theta \subset \mathbb{R}^{n_\theta}$ describes the parametric uncertainty, and $\mathbb{W} \subset \mathbb{R}^{n_x}$ bounds the additive disturbance. Both uncertainties may vary arbitrarily within their respective sets and are kept separate to emphasize their distinct roles in tube propagation. The sets \(\mathbb X\), \(\mathbb U\), \(\mathbb W\), and \(\Theta\) are assumed nonempty, compact, convex and 
to contain the origin. Moreover, \(\Theta\) is polyhedral and therefore admits the vertex representation
\(
\Theta = \operatorname{conv}\!\big(\{\theta_k\}_{k=1}^{\pe}\big)
\). For the remainder of this section, let \( z := (x,u,\theta) \in \mathcal Z := \mathbb{X}\times\mathbb{U}\times\Theta \), and use $f(z)$ as a shorthand for $f(x,u,\theta)$.

\begin{defn}[DC decomposition]
\label{def:DC_function}
A vector map $f:\mathcal{Z}\to\mathbb{R}^{n_x}$ is said to admit a  DC decomposition on $\mathcal{Z}$ if each component $f_i$ can be written as
\[
f_i(z) = g_i(z) - h_i(z), \qquad i\in\iv{1,n_x},
\]
where $g_i,h_i:\mathcal{Z}\to\mathbb{R}$ are finite-valued convex functions. In this case, we write $f=g-h$, where the equality is intended componentwise.
\end{defn}

The DC assumption is mild on compact convex domains: for instance, any twice continuously differentiable function admits a DC decomposition on such sets~\cite{horst1999dc}. Since DC decompositions are generally not unique, the resulting convexification and the associated conservatism depend on the particular decomposition selected.

Since \(g_i\) and \(h_i\) are finite and convex on \(\mathcal Z\), their subdifferentials are nonempty on \(\operatorname{ri}(\mathcal Z)\). Thus, for any fixed reference point \(\bar z\in\operatorname{ri}(\mathcal Z)\), select \( 
s_{g_i}(\bar z)\in\partial g_i(\bar z)\), \(\quad s_{h_i}(\bar z)\in\partial h_i(\bar z)\) for all \(i\in\iv{1,n_x}\), and define the corresponding affine minorants as
\begin{align*}
g_i^L(z;\bar{z}) &\coloneqq g_i(\bar{z}) + s_{g_i}(\bar{z})^\top (z-\bar{z}),\\
h_i^L(z;\bar{z}) &\coloneqq h_i(\bar{z}) + s_{h_i}(\bar{z})^\top (z-\bar{z});
\end{align*}
if $g_i$ and $h_i$ are differentiable at $\bar{z}$, one may take $s_{g_i}(\bar{z})=\nabla g_i(\bar{z})$ and $s_{h_i}(\bar{z})=\nabla h_i(\bar{z})$. Then, it is possible to obtain a convex upper bound on the projection of $f(z)$ along any direction $c\in\mathbb{R}^{n_x}$, as formalized next (cf.~\cite[Prop.~3]{FIACCHINI2010}).

\begin{lem}
\label{lem:c-bound}
Let $f$ admit a DC decomposition as per Definition~\ref{def:DC_function} and fix $\bar{z}\in\operatorname{ri}(\mathcal{Z})$. For any $c\in\mathbb{R}^{n_x}$, define
\begin{equation}
\label{eq:f_tilde}
\begin{aligned}
\widetilde{f}(z;\bar{z},c)
:={}&
\sum_{i\in \mathcal I^+(c)} c_i \big(g_i(z)-h_i^L(z;\bar{z})\big) \\
&+
\sum_{i\in \mathcal I^-(c)} c_i \big(g_i^L(z;\bar{z})-h_i(z)\big),
\end{aligned}
\end{equation}
where
\begin{equation}\label{eq:sign_set_defn}
\begin{split}
\mathcal I^+(c)&\coloneqq\{i\in\iv{1,n_x}\mid c_i\ge 0\},\\
\mathcal I^-(c)&\coloneqq\{i\in\iv{1,n_x}\mid c_i<0\}.
\end{split}    
\end{equation}
Then:
\begin{enumerate}[label=(\roman*)]
    \item the function $\widetilde{f}(\cdot\,;\bar{z},c)$ is convex on $\mathcal{Z}$;
    \item the  directional upper bound
    \begin{equation}\label{eq:lem1_UB}
    c^\top f(z)\le \widetilde{f}(z;\bar{z},c),
    \end{equation}
    holds for all $z\in\mathcal{Z}$.
\end{enumerate}
\end{lem}

\begin{pf}
Convexity of $h_i$ implies $h_i^L(z;\bar{z})\le h_i(z)$, for all $z\in\mathcal{Z}$, and
\begin{equation}\label{eq:lemma1_proof_a}
f_i(z) = g_i(z)-h_i(z)\le g_i(z)-h_i^L(z;\bar{z}),
\end{equation}
where the right-hand side is convex since $g_i(\cdot)$ is convex and $h_i^L(\cdot\,;\bar{z})$ is affine. Take $i\in \mathcal I^+(c)$. Since $c_i\ge 0$, multiplying by $c_i$ preserves the inequality in \eqref{eq:lemma1_proof_a}, as well as convexity of the right-hand side. 

Similarly, convexity of $g_i$ implies
$g_i^L(z;\bar{z})\le g_i(z)$, for all $z\in\mathcal{Z}$,
hence
\begin{equation}\label{eq:lemma1_proof_b}
g_i^L(z;\bar{z})-h_i(z)\le g_i(z)-h_i(z)=f_i(z).    
\end{equation}
Let $i\in \mathcal I^-(c)$. Then, multiplying both sides of \eqref{eq:lemma1_proof_b} by $c_i$ reverses the inequality, yielding
\[
c_i f_i(z)\le c_i\big(g_i^L(z;\bar{z})-h_i(z)\big),
\]
where the right-hand side is convex since
$g_i^L(\cdot\,;\bar{z})-h_i(\cdot)$ is concave and $c_i<0$.

The proof is concluded by noticing that the sum over all $i\in \mathcal I^+(c)\, \cup\, \mathcal I^-(c) = \iv{1,n_x}$ maintains the above properties.
\end{pf}
\begin{rem} 
Lemma~\ref{lem:c-bound} uses a common reference point \(\bar z\). Different fixed reference points may instead be selected across components, and separately for \(g_i\) and \(h_i\). Tighter bounds may also be obtained from direction-dependent DC decompositions by considering \( c^\top f(z)=g_c(z)-h_c(z) \). These choices are made offline and remain fixed in the online problem; their systematic selection is left for future research.
\end{rem}

\subsection{Robust Control Tubes}
\label{subsec:rct}

We now recall the set-based framework by which we formulate a robust description of the uncertain propagation of the system dynamics.
For a compact set $X\subseteq\mathbb{R}^{n_x}$, we let $\mathcal{F}(X)$ denote the collection of all sets $X^+$ that contain the one-step robust successors of $X$ through the dynamics \eqref{eq:sys} under admissible control actions:
\begin{equation}
\label{eq:onestepforward}
\mathcal{F}(X)
:=
\left\{
X^+\subseteq\mathbb{R}^{n_x}\ \middle|\
\begin{aligned}
&\forall x\in X,\ \exists u\in\mathbb{U}:\\
&f(x,u,\theta)+w\in X^+,\\
&\forall \theta\in\Theta,\ \forall w\in\mathbb{W}\\
\end{aligned}
\right\}.
\end{equation}

\begin{defn}[Robust control tube]
\label{def:rct}
A sequence of sets $\{X_0,\ldots,X_N\}$, with $X_k\subseteq\mathbb{X}$ for all $k\in\iv{0,N}$, is a \emph{robust control tube (RCT)} if
\begin{equation}
\label{eq::RCT}
X_{k+1}\in\mathcal{F}(X_k),\qquad k\in\iv{0,N\!-\!1}.
\end{equation}
A set $X\subseteq\mathbb{X}$ is \emph{robust control invariant (RCI)} if $\{X,X\}$ is an RCT.
\end{defn}
The above definition leads directly to a robust set-valued MPC formulation, where uncertain state propagation is controlled according to given optimality criteria. In particular, let \(L_0\), \(L\), and \(L_N\) be real-valued functionals of the tube cross-sections defining the initial, stage, and terminal costs, respectively, and let \(\Omega\) denote the terminal set. The canonical finite-horizon tube MPC problem reads as
\begin{equation}
\label{eq:TMPC}
\begin{aligned}
\min_{X_0,\ldots,X_N}\quad &
L_0(X_0)+\sum_{k=0}^{N-1}L(X_k)+L_N(X_N)\\
\text{s.t.}\quad
&x\in X_0,\\
&X_{k+1}\in\mathcal{F}(X_k),\quad k\in\iv{0,N\!-\!1},\\
& X_{N} \subseteq \Omega,\\
&X_k\subseteq\mathbb{X},\qquad\qquad k\in\iv{0,N}.
\end{aligned}
\end{equation}

We note that feasibility of \eqref{eq:TMPC} implies, by \eqref{eq:onestepforward}, the existence of an input sequence that robustly regulates the dynamics~\eqref{eq:sys} from the initial state $x$ to the terminal set \(\Omega\). However, problem~\eqref{eq:TMPC} is intractable in general, since its decision variables are sets and the one-step propagation in~\eqref{eq:onestepforward} involves nonlinear robust reachability. The approach developed in this paper addresses numerical tractability by:
\begin{enumerate}[label=(\roman*)]
    \item replacing the original nonlinear dynamics by a convex outer approximation that yields a sufficient one-step tube propagation certificate;
    \item restricting tube cross-sections to a family of polytopes with fixed combinatorial structure, so that set propagation can be efficiently encoded in finite-dimensional form.
\end{enumerate}
In Section~\ref{sec:convex_reformulation}, we build on the result of Lemma~\ref{lem:c-bound} to obtain a convex envelope of the uncertain dynamics. Below, we introduce the main ingredient of this envelope.

\subsection{Configuration‑Constrained Polytopes}
\label{subsec:ccp}

We parameterize tube cross-sections as polytopes of the form
\begin{equation}
\label{eq:CCPolytope}
X(y)\coloneqq \{x\in\mathbb{R}^{n_x}\mid Fx\le y\},
\end{equation}
where $F\in\mathbb{R}^{\fe\times n_x}$ is a fixed \emph{template} matrix and $y\in\mathbb{R}^{\fe}$ is a manipulable parameter. We assume throughout that
\(
Fx\le 0  \implies  x=0,
\)
or equivalently \(X(0)=\{0\}\), which ensures that \(X(y)\) is bounded whenever it is nonempty. 

The parameterization~\eqref{eq:CCPolytope} alone does not guarantee that the combinatorial face structure of $X(y)$ remains unchanged as $y$ varies. Following~\cite{CCTMPC_VILLANUEVA2024111543}, we therefore restrict $y$ to a set of admissible values for which the face lattice is preserved.

\begin{defn}[Configuration cone]
The polyhedral cone
\[
\mathcal{E}\coloneqq \{y\in\mathbb{R}^{\fe}\mid Ey\le 0\}
\]
is called a \emph{configuration cone} for the family~\eqref{eq:CCPolytope} if the polytope $X(y)$ has the same face configuration for all $y\in\mathcal{E}$.
\end{defn}
The position of facets and vertices depends \emph{affinely} on $y$ when the latter is restricted to a configuration cone: in particular, there exist matrices $\boldsymbol{W}\coloneqq \{W_j\}_{j=1}^{\ve}$ such that
\begin{equation}
\label{eq:vertex_config}
X(y)=\operatorname{conv}(\{W_j y\}_{j=1}^\ve),\qquad \forall y\in\mathcal{E}.
\end{equation}
Thus, we refer to $\langle F,E,\boldsymbol{W}\rangle$ as a \emph{configuration triple}. Its construction is nontrivial; procedures for generating such triples and trading geometric flexibility for reduction in complexity are discussed in~\cite{CCTMPC_VILLANUEVA2024111543}. In Appendix~\ref{app:template_dc}, we present an offline NLP procedure for template initialization through sign-preserving deformation, which extends the procedure in~\cite[\S IV-B]{badalamenti2025_efficientCCTMPC} from linear difference inclusions to the present DC setting.

\section{Convex Reformulation for Tube Propagation}
\label{sec:convex_reformulation}

Here we specialize Lemma~\ref{lem:c-bound} to the directions defined by the facet normals of the polytope family~\eqref{eq:CCPolytope}. In particular, we construct an upper bound of the form \eqref{eq:lem1_UB} for each of the directions encoded by the rows of $F$, considering a fixed linearization point $\bar{z}:=(\bar{x}, \bar{u}, \bar{\theta})\in\mathbb{X}\times\mathbb{U}\times\Theta$.

Recall 
\(
F=\begin{bmatrix}F_1^\top \cdots\,  F_\fe^\top \end{bmatrix}^\top\in\mathbb{R}^{\fe\times n_x},
\)
where each row $F_l\in\mathbb{R}^{1\times n_x}$ defines the normal to facet $l$. For $(x,u,\theta)\in\mathbb{X}\times\mathbb{U}\times\Theta$, define
\begin{equation}
\label{eq:upperbound_polytope}
        \widetilde{f}_F(x, u,\theta) := 
        \begin{bmatrix}
        \widetilde{f}((x, u, \theta); (\bar{x}, \bar{u}, \bar{\theta}), F_1^\top) \\
        \widetilde{f}((x, u, \theta); (\bar{x}, \bar{u}, \bar{\theta}), F_2^\top) \\
        \vdots \\
        \widetilde{f}((x, u, \theta); (\bar{x}, \bar{u}, \bar{\theta}), F_\fe^\top)
        \end{bmatrix}
        \in \R^\fe.
    \end{equation}
Note that $(\bar{x}, \bar{u}, \bar{\theta})$ is fixed at design time. To ease the notation, we consider it implicit throughout the rest of the paper. 
The following result is an immediate consequence of Lemma~\ref{lem:c-bound}.

\begin{lem}
\label{prop:convex-upperbound} 
Each component of $\widetilde f_F(x,u,\theta)$ is convex in $(x,u,\theta)$ on $\mathbb{X}\times\mathbb{U}\times\Theta$, and
\begin{equation}
\label{eq:Ff_upperbound}
Ff(x,u,\theta)\le \widetilde f_F(x,u,\theta),
\quad \forall (x,u,\theta)\in\mathbb{X}\times\mathbb{U}\times\Theta,
\end{equation}
where the inequality holds componentwise.
\end{lem}

Lemma~\ref{prop:convex-upperbound} yields a finite-dimensional sufficient certificate for robust one-step propagation. Define the support vector \(d\in\mathbb R^\fe\) of the disturbance set along the directions defined by the rows of \(F\) as
\[
d_l\coloneqq \max_{w\in\mathbb{W}}F_l w,\qquad l\in\iv{1,\fe}.
\]

\begin{prop}
\label{prop:surrogate_implies_original}
Let \(\Theta=\operatorname{conv}(\{\theta_k\}_{k=1}^{\pe})\) and let \(X(y)\) be nonempty. Suppose that, for every \(x\in X(y)\), there exists \(u\in\mathbb{U}\) such that
\begin{equation}
\label{eq:finite_prop_certificate}
\widetilde f_F(x,u,\theta_k)+d\le y^+,\qquad \forall k\in\iv{1,\pe}.
\end{equation}
Then \(X(y^+)\in\mathcal F(X(y))\).
\end{prop}

\begin{pf}
Fix \(x\in X(y)\), and let \(u\in\mathbb{U}\) satisfy~\eqref{eq:finite_prop_certificate}. Since each component of \(\widetilde f_F(x,u,\theta)\) is convex in \(\theta\), 
Jensen's inequality gives
\[
\widetilde f_F(x,u,\theta)+d\le y^+,\qquad \forall \theta\in\Theta .
\]
Recall \(Fw\le d\) for all \(w\in\mathbb{W}\). Hence
\[
\widetilde f_F(x,u,\theta)+Fw\le y^+,\qquad \forall \theta\in\Theta,\ \forall w\in\mathbb{W} .
\]
Using Lemma~\ref{prop:convex-upperbound},
\[
F(f(x,u,\theta)+w)\le y^+,\qquad \forall \theta\in\Theta,\ \forall w\in\mathbb{W},
\]
and therefore \(f(x,u,\theta)+w\in X(y^+)\). Since, by hypothesis, for every \(x\in X(y)\) there exists an admissible input \(u\in\mathbb U\) for which all uncertain successors belong to \(X(y^+)\), it follows from \eqref{eq:onestepforward} that 
\(X(y^+)\in\mathcal F(X(y))\).
\end{pf}
Thus, robust propagation is guaranteed by the finite inequalities~\eqref{eq:finite_prop_certificate}. Next, we qualify the existence of an admissible state-dependent input by exploiting the vertex representation of configuration-constrained polytopes.

\subsection{Configuration-Constrained RCTs}
\label{subsec:cc_rct_final}

In this section, we formulate a finite-dimensional condition implying~\eqref{eq:finite_prop_certificate}, provided that \(X(y)\) is nonempty.

Recall that, by convexity, an admissible input for an arbitrary \(x\in X(y)\) is obtained as the convex combination of the inputs corresponding to the vertices of $X(y)$. Let
\(
\mathbf{u}\coloneqq \left(u_1,\ldots,u_\ve\right)
\in\mathbb{R}^{n_u\ve}
\)
collect the inputs assigned to the $\ve$ vertices of $X(y)$, and define the selector matrices
\[
U_j\coloneqq e_j^\top\otimes \I_{n_u}\in\mathbb{R}^{n_u\times n_u\ve},
\qquad j\in\iv{1,\ve},
\]
where $e_j$ is the $j$-th standard basis in $\mathbb{R}^\ve$, so that $U_j \mathbf{u}=u_j$. We introduce the set
\(
\widetilde{\mathbb S}\subseteq
\mathbb{R}^{\fe}\times\mathbb{R}^{n_u\ve}\times\mathbb{R}^{\fe}
\) to define the one-step robust propagation condition for sets parameterized as configuration-constrained polytopes:
\begin{align}
\label{eq:S_tilde_definition}
    \widetilde{\mathbb{S}}\coloneqq
    \left\{ \begin{pmatrix}
        y\\\mathbf{u}\\y^+ 
    \end{pmatrix} \, \middle| \,  
    \begin{aligned}
    &\forall (j,k)\in\iv{1,\ve}\times\iv{1,\pe},
    \vspace{3pt}\\
    &\widetilde{f}_F(W_jy,U_j\mathbf{u},\theta_k)+d \leq y^+,\vspace{3pt}\\
    & Ey \leq 0, \ W_j y \in \mathbb{X}, \ U_j \mathbf{u} \in \mathbb{U}
    \end{aligned} \right\}.
\end{align}
By construction, the set \(\widetilde{\mathbb S}\) is convex in $\Svec$, since it is defined by convex inequalities and linear constraints. Then, the following result holds.

\begin{lem}
\label{lem:RCT_CC_DC}
If $\Svec\in \widetilde{\mathbb S}$, then $X(y^+)\in \mathcal F(X(y))$.
\end{lem}

\begin{pf}
Let $x\in X(y)$. Since $y\in\mathcal{E}$, there exists $\{W_j\}_{j=1}^{\ve}$ such that $X(y)=\operatorname{conv}(\{W_j y\}_{j=1}^\ve)$, cf.~\eqref{eq:vertex_config}. Then by Carath\'eodory's theorem there exists $\lambda(x)\in\Delta^{\ve-1}$ such that $x=\sum_{j=1}^{\ve}\lambda_j(x)W_j y$. Define 
\begin{equation}
\label{eq:vertex_interpolated_input}
\mu(x)\coloneqq\sum_{j=1}^{\ve}\lambda_j(x)\,U_j \mathbf{u}.
\end{equation}
Since $U_j\mathbf{u}\in\mathbb{U}$ for all $j\in\iv{1,\ve}$ and $\mathbb{U}$ is convex, it follows that $\mu(x)\in\mathbb{U}$. 

For every \(k\in\iv{1,\pe}\), convexity of \(\widetilde f_F(\cdot,\cdot,\theta_k)\) gives
\[
\widetilde f_F(x,\mu(x),\theta_k) \le \sum_{j=1}^{\ve}\lambda_j(x)\widetilde f_F(W_jy,U_j\mathbf u,\theta_k).
\]
Using \((y,\mathbf u,y^+)\in\widetilde{\mathbb S}\), the above inequality implies
\[
\widetilde f_F(x,\mu(x),\theta_k)+d\le y^+,\qquad \forall k\in\iv{1,\pe}.
\]
Thus the hypothesis of Proposition~\ref{prop:surrogate_implies_original} holds with \(u=\mu(x)\) for every \(x\in X(y)\), and the claim follows.
\end{pf}

Lemma~\ref{lem:RCT_CC_DC} provides a tractable sufficient certificate for robust one-step tube propagation. The converse, however, need not hold: a state-dependent admissible input satisfying~\eqref{eq:finite_prop_certificate} might not be expressed as convex combination of an admissible collection of vertex inputs (i.e., satisfying ~\eqref{eq:S_tilde_definition}).
Hence, the proposed finite-dimensional reformulation may introduce conservatism in addition to that arising from the DC-based outer approximation.

Lemma~\ref{lem:RCT_CC_DC} also provides a finite-dimensional sufficient condition for robust control invariance (cf.~Definition~\ref{def:rct}).

\begin{cor}[CC-RCI certificate]
\label{cor:RCI_CC}
If there exists $\mathbf{u}$ such that $\RCIvec\in\widetilde{\mathbb S}$, then the configuration-constrained polytope $X(y)$ is an RCI set for system~\eqref{eq:sys}.
\end{cor}

\begin{pf}
If $\RCIvec\in\widetilde{\mathbb S}$, Lemma~\ref{lem:RCT_CC_DC} gives \(X(y)\in\mathcal F(X(y))\), which recovers the definition of robust control invariance.
\end{pf}

\section{Nonlinear CC-TMPC for Regulation}
\label{sec:CC-TMPC}

We are now ready to develop a tube MPC design for regulation, based on the one-step propagation certificate derived in the previous section.
Let the sequences $\{y_0,\ldots,y_N\}$ and $\{\bfu_0,\ldots,\bfu_{N}\}$ parameterize the RCT along the $N$-step prediction horizon.
The design goal is to achieve convergence of the RCT to a predefined CC-RCI set. We configure the latter by a pair $(y^\circ, \mathbf{u}^\circ)$ which satisfies the RCI condition. This can be chosen, for instance, as the minimizer of some given convex criterion $\omega\colon \R^\fe \times\R^{n_u\ve}\to \R$, i.e.,
\begin{equation}
 \label{eq:OptRCI_set}
(y^\circ,\mathbf{u}^\circ)\in \arg\min_{y,\mathbf{u}} \; \omega(y,\mathbf{u}) \quad \text{s.t.} \quad  \RCIvec \in \widetilde{\mathbb{S}},
\end{equation}
where \(\widetilde{\mathbb S}\), defined in~\eqref{eq:S_tilde_definition}, encodes the one-step robust propagation condition together with the state, input, and configuration constraints. Note that the target $(y^\circ,\mathbf{u}^\circ)$ is fixed throughout: as such, it can be computed offline.

Fix a contraction factor \(\gamma\in[0,1)\). Then, for a given state \(x\in\mathbb X\), the proposed configuration-constrained tube MPC problem takes the form (cf.~\eqref{eq:TMPC})
\begin{subequations}
\label{eq:CCTMPC_new}
\begin{align}
\min_{\substack{\{y_0,\ldots,y_N\},\\ \{\mathbf{u}_0,\ldots,\mathbf{u}_{N}\}}} &\ \sum_{k=0}^{N-1} 
    L(y_k,\bfu_k)
    + L_N(y_N, \bfu_N) \label{eq:CCTMPC_cost}\\
    \text{s.t.}\hspace{5pt} &\  F{x} \leq y_0, \label{eq:CCTMPC_init_constr}\\
    &\ (y_k,\mathbf{u}_k,y_{k+1}) \in \widetilde{\mathbb{S}},\quad  k\in\iv{0,N\!-\!1}, \label{eq:CCTMPC_tube_constr}\\
    &\ (y_N,\mathbf{u}_N,\gamma y_N + (1-\gamma)y^\circ) \in \widetilde{\mathbb{S}}.\label{eq:CCTMPC_term_constr}
\end{align}
\end{subequations}
Its feasible set is denoted by \(\mathbb X_N\coloneqq\{x\in\mathbb X\mid \eqref{eq:CCTMPC_new}\text{ is feasible}\}\). Here, $L\colon \R^\fe \times\R^{n_u\ve}\to \R$ is a convex, positive-definite regulation cost that vanishes at the optimal RCI parameters $(y^\circ,\mathbf{u}^\circ)$.
We use a terminal cost $L_N: \R^{\fe}\times\R^{n_u\ve}\to\R$ together with the contractive one-step condition \eqref{eq:CCTMPC_term_constr}. This yields an implicit terminal set for the chosen tube parameterization, i.e., $\Omega\subseteq\mathbb{R}^\fe$. This induces a closed-loop Lyapunov descent property in the tube configuration space, ensuring convergence of the RCT to the optimal RCI set.
We discuss this in detail in the next section, focusing on the choice
\begin{align*}
L(y,\mathbf{u}) & \coloneqq
    \norm{(y-y^\circ,\mathbf{u}-\mathbf{u}^\circ)}_Q^2\\
L_N(y,\mathbf{u})  & \coloneqq
    \norm{(y-y^\circ,\mathbf{u}-\mathbf{u}^\circ)}_P^2,
\end{align*}
with $Q,P\succ0$.\footnote{The results can be extended with minor modifications to more general definitions. The interested reader is referred to \cite{Villanueva2020}.}

\subsection{Implicit Terminal Set}

We now discuss how the terminal cost $L_N$, together with the constraint \eqref{eq:CCTMPC_term_constr}, produce an implicit terminal set $\Omega\subseteq\mathbb{R}^\fe$ where desired Lyapunov-type properties are satisfied for the RCT parameters. 
Define 
\begin{subequations}\label{eq:cost_to_travel}
\begin{align}
V(y,y^+)\coloneqq &\min_{\mathbf{u}} \,   L(y,\mathbf{u}) \\ 
& \ \text{s.t.}  \ \ \Svec \in \widetilde{\mathbb{S}}. \label{eq:cost_to_travel_b}
\end{align}
\end{subequations}
Whenever \eqref{eq:cost_to_travel} is infeasible, set \(V(y,y^+)=+\infty\); the same convention is adopted in the next definitions.
Also,
\begin{subequations}\label{eq:terminal_problem}
\begin{align}
    T(y):=&\min_{\mathbf{u}} \,   L_N(y,\bfu)
    \label{eq:terminal_problem_a}\\
    & \hspace{5pt} \text{s.t.} \ \  (y,\mathbf{u},\gamma y + (1-\gamma)y^\circ) \in \widetilde{\mathbb{S}}. 
\end{align}
\end{subequations}
Let \(\Omega\coloneqq \big\{ y  \mid  T(y) < +\infty \big\}\) and note that feasibility of \eqref{eq:OptRCI_set} implies that $\Omega\neq\emptyset$ for any $\gamma\in[0,1)$.

With these definitions in hand, problem \eqref{eq:CCTMPC_new} can be rewritten as
\begin{subequations}\label{eq:CCTMPC_CTT}
\begin{align}
\min_{\{y_0,\ldots,y_N\}} &\ \sum_{k=0}^{N-1} V(y_k,y_{k+1}) + T(y_N) \label{eq:CCTMPC_CTT_a}\\
\text{s.t.}\hspace{3pt} &\ F x \leq y_0,\label{eq:CCTMPC_CTT_b} 
\end{align}
\end{subequations}
where the first term in \eqref{eq:CCTMPC_CTT_a} is known as the \emph{cost-to-travel} function~\cite{Villanueva2020}. Next, we show that the function \(T(y_N)\) satisfies the Lyapunov descent condition over $\Omega$.

\begin{lem}
\label{lem:terminal_decrease}
Let $\gamma\in[0,1)$ and $Q,P\succ 0$ satisfy $Q+\gamma^2 P \preceq P$. The following holds: 
\begin{enumerate}[label=(\roman*)]
    \item the function \(T(y)\) is null at $y=y^\circ$;
    \item for any \(\bar{y} \in \Omega\), 
    \begin{equation}
        \label{eq:descent_condition}
        \min_{y^+} \ T(y^+)+ V(\bar{y},y^+) \leq T(\bar{y}). 
    \end{equation}
\end{enumerate}
\end{lem}

\begin{pf} We first note that for $y = y^\circ$, the feasible set of \eqref{eq:terminal_problem} coincides with that of \eqref{eq:OptRCI_set}; then, given the choice of $L_N$, $\mathbf{u} = \mathbf{u}^\circ$ solves \eqref{eq:terminal_problem}, implying \textit{(i)}. 
Now, for \textit{(ii)}, let \(\bar{\mathbf{u}}\) be the minimizer of \eqref{eq:terminal_problem} at $y=\bar{y}$ (which exists since $\bar{y}\in\Omega$).
Set $\bar{y}^+ \coloneqq \gamma\bar{y}+(1-\gamma)y^\circ$. Notice that by feasibility 
of~\eqref{eq:terminal_problem}, $y=\bar{y}$, $\mathbf{u}=\bar{\mathbf{u}}$ and $y^+=\bar{y}^+$ must be also feasible for \eqref{eq:cost_to_travel_b}. Therefore,
\begin{equation}\label{eq:lhs_P2}
    V(\bar{y},\bar{y}^+) \leq  \norm{(\bar{y}-y^\circ, \bar{\mathbf{u}}-\mathbf{u}^\circ)}_Q^2;
\end{equation}
moreover, it follows from~\eqref{eq:S_tilde_definition} that
\begin{equation}\label{eq:thm1_proof_a}
\widetilde{f}_F(W_j\bar{y},U_j\bar{\mathbf{u}},\theta_k) + d \ \le\ \bar{y}^+
\end{equation}
holds for all $(j,k)\in\iv{1,\ve}\times\iv{1,\pe}$. 
Let $\bar{\mathbf{u}}^+ := \gamma\bar{\mathbf{u}} + (1-\gamma)\mathbf{u}^\circ$. Using the definitions of $\bar{y}^+$ and $\bar{\mathbf{u}}^+$, and convexity of $\widetilde{f}_F$, we obtain
\begin{align*}
\widetilde{f}_F\big(W_j\bar{y}^+,{}& U_j\bar{\mathbf{u}}^+,\theta_k\big)+d \\
&\leq \gamma\big(\widetilde{f}_F(W_j\bar{y},U_j\bar{\mathbf{u}},\theta_k)+d\big)\\
& \quad +(1-\gamma)\big(\widetilde{f}_F(W_j y^\circ,U_j \mathbf{u}^\circ,\theta_k)+d\big)\\ 
&\leq \gamma \bar{y}^+ + (1-\gamma) y^\circ,   \quad \forall(j,k)\in\iv{1,\ve}\times\iv{1,\pe},
\end{align*}
where the last inequality follows from \eqref{eq:thm1_proof_a} and a similar application of \eqref{eq:S_tilde_definition} to \eqref{eq:OptRCI_set}.
Thus, \(\mathbf{u}=\bar{\mathbf{u}}^+\) is feasible for \eqref{eq:terminal_problem} at \(y=\bar{y}^+\), and
\begin{align*}
    T(\bar{y}^+)  &\le \| \begin{pmatrix}\bar{y}^+\!-y^\circ, \bar{\mathbf{u}}^+\!-\mathbf{u}^\circ \end{pmatrix}\|_P^2 \\
    & = \|\gamma\begin{pmatrix}\bar{y}-y^\circ, \bar{\mathbf{u}}-\mathbf{u}^\circ\end{pmatrix}\|_P^2\\
&= \gamma^2 \|\begin{pmatrix}\bar{y}-y^\circ, \bar{\mathbf{u}}-\mathbf{u}^\circ\end{pmatrix}\|_P^2.
\end{align*}
Finally, combining the latter with \eqref{eq:lhs_P2} we obtain 
\begin{align*}
\min_{y^+} \ & T({y}^+) + V(\bar{y},{y}^+) \\
  &\leq 
    T(\bar{y}^+) + V(\bar{y},\bar{y}^+) \\
    &\leq \gamma^2 \|\begin{pmatrix}\bar{y}-y^\circ, \bar{\mathbf{u}}-\mathbf{u}^\circ\end{pmatrix}\|_P^2 + \norm{(\bar{y}-y^\circ, \bar{\mathbf{u}}-\mathbf{u}^\circ)}_Q^2 \\
    & \leq \norm{(\bar{y}-y^\circ,\bar{\mathbf{u}}-\mathbf{u}^\circ)}_P^2 = T(\bar{y}),
\end{align*}
where the last inequality follows from \(Q+\gamma^2 P\preceq P\). This recovers \eqref{eq:descent_condition} and concludes the proof.
\end{pf}

\subsection{Closed-Loop Control Law}

Denote the optimizers of~\eqref{eq:CCTMPC_new} by \(\{y_k^*(x)\}_{k=0}^{N}\) and \(\{\mathbf u_k^*(x)\}_{k=0}^{N}\). Since \(Fx\leq y_0^*(x)\), we have \(x\in X(y_0^*(x))\), and therefore there exists \(\lambda(x)\in\Delta^{\ve-1}\) such that
\[
x=\sum_{j=1}^{\ve}\lambda_j(x)W_jy_0^*(x).
\]
By Lemma~\ref{lem:RCT_CC_DC}, any such interpolation can be used to construct an admissible feedback input preserving the certified one-step robust inclusion (given the existence of the feasible vertex sequence $\{\mathbf u_k^*(x)\}_{k=0}^{N}$). Since the representation is not unique, any selection policy satisfying the simplex and reconstruction constraints is admissible. Here, we select the one yielding the smallest control effort through the QP
\begin{equation}
\label{eq:lambda_selection}
\begin{aligned}
\lambda^*(x)\in\arg\min_{\lambda\in\Delta^{\ve-1}}\quad &
\Bigg\|\sum_{j=1}^{\ve}\lambda_jU_j\mathbf u_0^*(x)\Bigg\|_R^2\\
\text{s.t.}\quad &
x=\sum_{j=1}^{\ve}\lambda_jW_jy_0^*(x),
\end{aligned}
\end{equation}
with \(R\succ0\). Thus, we define the feedback law as
\begin{equation}
\label{eq:MPC_control_law}
\mu(x):=\sum_{j=1}^{\ve}\lambda_j^*(x)U_j\mathbf u_0^*(x).
\end{equation}
Hence, the dynamics~\eqref{eq:sys} in closed-loop read as
\begin{equation}\label{eq:closed-loop}
x_{t+1}=f(x_t,\mu(x_t),\theta_t)+w_t,
\end{equation}
with $\theta_t \in \Theta$, $w_t \in \mathbb{W}$. For \(x\in\mathbb X_N\), let \(\mathcal O(x)\) denote the optimal value of Problem~\eqref{eq:CCTMPC_new}. We are now ready to establish recursive feasibility and robust convergence.

\begin{thm}
\label{thm::stability}
Let \(Q\succ0\), \(P\succ0\), and \(\gamma\in[0,1)\) satisfy \(Q+\gamma^2P\preceq P\). Suppose that \((y^\circ,\mathbf u^\circ,y^\circ)\in\widetilde{\mathbb S}\), and that \(x_0\in\mathbb X_N\). Then, the closed-loop dynamics~\eqref{eq:closed-loop} satisfy the following properties:
\begin{enumerate}[label=(\roman*)]
\item Problem~\eqref{eq:CCTMPC_new} remains feasible for all \(t\in\mathbb N\);
\item \(\lim_{t\to\infty}\mathcal{O}(x_t)=0\), and
\[
\lim_{t\to\infty}(y_k^*(x_t),\bfu_k^*(x_t))=(y^\circ,\bfu^\circ),\qquad k\in\iv{0,N}.
\]
\end{enumerate}
\end{thm}

\begin{pf}
$(i)$ 
Let \(\{y_k^*(x_t),\mathbf u_k^*(x_t)\}_{k=0}^{N}\) be an optimizer of~\eqref{eq:CCTMPC_new}, for $x=x_t$, and define the shifted candidate sequence as
\begin{subequations}\label{eq:shifted_seqs}
\begin{align} 
&\!\left\{
\begin{aligned}
\hat y_k &=y_{k+1}^*(x_t),\\
\hat{\mathbf u}_k &=\mathbf u_{k+1}^*(x_t),
\end{aligned}\right.  &&k\in\iv{0,N\!-\!1}, \label{eq:shifted_sequence}\\[1mm] 
&\!\left\{
\begin{aligned}
  \hat y_N &=\gamma y_N^*(x_t)+(1-\gamma)y^\circ,\\ \hat{\mathbf u}_N &=\gamma\mathbf u_N^*(x_t)+(1-\gamma)\mathbf u^\circ.  
\end{aligned}\right.
\label{eq:shifted_terminal_pair} 
\end{align}
\end{subequations}
Since \(x_t\in X(y_0^*(x_t))\) and \((y_0^*(x_t),\mathbf u_0^*(x_t),y_1^*(x_t))\in\widetilde{\mathbb S}\), Lemma~\ref{lem:RCT_CC_DC} with~\eqref{eq:MPC_control_law} imply $X(y_1^*(x_t))\in\mathcal{F}(X(y_0^*(x_t)))$, hence \(x_{t+1}\in X(y_1^*(x_t))\) (or, equivalently, \(Fx_{t+1}\le y_1^*(x_t)\)).
Then, with the choice in \eqref{eq:shifted_seqs}, constraint \eqref{eq:CCTMPC_init_constr} holds for $x = x_{t+1}$, since $x_{t+1}\in X(\hat{y}_0)$. 
The tube constraints \eqref{eq:CCTMPC_tube_constr} for \(k\in\iv{0,N\!-\!2}\) are inherited from the previous optimal sequence; for $k=N-1$, \eqref{eq:CCTMPC_tube_constr} follows from~\eqref{eq:shifted_terminal_pair} and~\eqref{eq:CCTMPC_term_constr}.
Moreover, \begin{multline} (\hat y_N,\hat{\mathbf u}_N, \gamma\hat y_N+(1-\gamma)y^\circ) =\\ \gamma (y_N^*(x_t),\mathbf u_N^*(x_t),\hat y_N) + (1-\gamma) (y^\circ,\mathbf u^\circ,y^\circ) \in\widetilde{\mathbb S}, \end{multline}
where the last inclusion is due to \eqref{eq:CCTMPC_term_constr} and convexity of \(\widetilde{\mathbb S}\). Thus, given feasibility of \eqref{eq:CCTMPC_new} at $x = x_t$, $t\in\mathbb{N}$, the shifted sequence is feasible at \(x = x_{t+1}\), proving recursive feasibility.

$(ii)$
Evaluating the cost function in~\eqref{eq:CCTMPC_CTT_a} using the shifted candidate gives the upper bound (note that \eqref{eq:CCTMPC_CTT_a} is the same as \eqref{eq:CCTMPC_cost})
\[
\scalemath{0.87}{
\mathcal O(x_{t+1})
\le
\sum_{k=1}^{N-1}V(y_k^*(x_t),y_{k+1}^*(x_t))
+V(y_N^*(x_t),\hat y_N)+T(\hat y_N).
}
\]
By Lemma~\ref{lem:terminal_decrease}, $V(y_N^*(x_t),\hat y_N)+T(\hat y_N) \le T(y_N^*(x_t))$,
therefore
\[
\mathcal O(x_{t+1})
\le
\mathcal O(x_t)-V(y_0^*(x_t),y_1^*(x_t)).
\]
By definition of $V$ in \eqref{eq:cost_to_travel}, and since \(\mathcal O(x)\ge0\), $\forall x$, the above implies that \(\mathcal O\) is nonincreasing and convergent along the system's trajectory (due to continuity of $L$ and $L_N$). Hence
\[
\lim_{t\to\infty}V(y_0^*(x_t),y_1^*(x_t))=0,
\]
which in turn, by definition of \(L(y,\bfu)\), implies
\[
(y_0^*(x_t),\mathbf u_0^*(x_t))\to(y^\circ,\mathbf u^\circ),
\]
and $\lim_{t\to\infty}\mathcal O(x_t)=0$.
From this (and $Q,P\succ0$) it follows that all other terms in~\eqref{eq:CCTMPC_cost} must vanish asymptotically, hence
\[
\lim_{t\to\infty}(y_k^*(x_t),\mathbf u_k^*(x_t))=(y^\circ,\mathbf u^\circ),\qquad \forall k\in\iv{0,N}.
\]
In other words, since $x_t\in X(y_0^*(x_t))$, the optimized tube section containing the closed-loop state trajectory converges to \(X(y^\circ)\). 
\end{pf}

\section{Numerical Example}
\label{sec:example}

We consider the forward-Euler discretization of the uncertain Duffing oscillator
\begin{equation}
\label{eq:duffing}
\begin{aligned} x_{t+1} &= x_t+T_s\left( f_{\mathrm{ct}}(x_t,u_t,\theta_t) + \begin{bmatrix}0\\w_t\end{bmatrix} \right),\\[1mm] f_{\mathrm{ct}}(x,u,\theta) &= \begin{bmatrix} x_2\\ -\delta x_2-\theta_1x_1-\beta x_1^3+\theta_2u \end{bmatrix}, \end{aligned} \end{equation}
where \(T_s=0.2\,\mathrm{s}\) is the sampling time, and \(\delta=0.2\), \(\beta=0.5\) are fixed model parameters. At each sampling instant, the uncertain parameters and additive disturbance satisfy $\theta_t\in\Theta=[0.8,1.2]\times[0.9,1.1]$ and $w_t\in\mathbb W=[-0.25,0.25]$, while the state and input constraints are $\mathbb X=[-2,0.6]\times[-2,2]$, $\mathbb U=[-1,1]$. 

The directional bounds are constructed around \((\bar x,\bar u,\bar\theta)=((0,0),0,(1,1))\). The uncertain bilinear terms are decomposed using the identity \(ab=\frac14(a+b)^2-\frac14(a-b)^2\); for the cubic restoring force, we use
\[
-\beta x_1^3=
\left(-\beta x_1^3+\frac{\rho_c}{2}x_1^2\right)-\frac{\rho_c}{2}x_1^2,
\quad
\rho_c=6\beta x_{1,\max}=1.8,
\]
where $\rho_c$ is the smallest constant ensuring convexity over the constrained domain. A continuous-time DC decomposition \(f_{\mathrm{ct}}=g_{\mathrm{ct}}-h_{\mathrm{ct}}\) is therefore obtained with
\begin{align*}
&\scalemath{0.85}{g_{\mathrm{ct}}(x,u,\theta) = \begin{bmatrix}
    x_2 \\
    -\delta x_2+\frac14(\theta_1-x_1)^2+\frac14(\theta_2+u)^2-\beta x_1^3+\frac{\rho_c}{2}x_1^2
\end{bmatrix},}\\
&\scalemath{0.85}{h_{\mathrm{ct}}(x,u,\theta) = \begin{bmatrix}
    0 \\
    \frac14(\theta_1+x_1)^2+\frac14(\theta_2-u)^2+\frac{\rho_c}{2}x_1^2
\end{bmatrix}.}
\end{align*}
\normalsize
The corresponding discrete-time DC maps are
\begin{displaymath}
\scalemath{0.9}{
g_d(x,u,\theta)=x + T_s g_{\mathrm{ct}}(x,u,\theta), \quad h_d(x,u,\theta)=T_sh_{\mathrm{ct}}(x,u,\theta),}
\end{displaymath}
and the additive disturbance set in the discrete-time model is \(\{0\}\times T_s\mathbb W\).

The seed template is a regular dodecagon with unit offsets and facet normals defined as
\[
\bar F_l=
\begin{bmatrix}
\cos\varphi_l & \sin\varphi_l
\end{bmatrix},
\;
\varphi_l=\frac{\pi+4\pi(l-1)}{24},
\quad
l\in\iv{1,12}.
\]
The sign-preserving transformation in Appendix~\ref{app:template_dc} yields \(F=\bar FT\), where
\[
T=
\begin{bmatrix}
1.8275 & 0.4225\\
0.1765 & 1.4528
\end{bmatrix}.
\]
The resulting configuration-constrained family \(X(y)=\{x\mid Fx\le y\}\) has \(\fe=\ve=\ee=12\).

The target pair \((y^\circ,\mathbf u^\circ)\) is computed from~\eqref{eq:OptRCI_set} using \(\omega(y,\mathbf u) = 100 \|y\|_2^2+\|\mathbf u\|_2^2\). For the online regulation cost, define \( \alpha_j\coloneqq \operatorname{blkd}\!\left( W_j-\ve^{-1}\sum_{i=1}^{\ve}W_i,\, U_j-\ve^{-1}\sum_{i=1}^{\ve}U_i \right)\); then, the stage-cost matrix is selected as 
\( Q=\sum_{j=1}^{\ve} \alpha_j^\top\operatorname{blkd}(10\I_2,1)\alpha_j+\I_{24}
\), and the terminal weight is \(P=(1-\gamma^2)^{-1}Q\), with \(\gamma=0.98\). The prediction horizon is \(N=3\). The feasible set \(\mathbb X_3\) is approximated using 100 support directions, and the convex hull of the resulting support points defines the feasible-region approximation \(\mathcal R(3)\). The relative area gap between the corresponding inner and outer approximations is \(0.223\%\). Six initial conditions are selected near the boundary of \(\mathcal R(3)\). Figure~\ref{fig:STATE_SPACE} reports the resulting closed-loop trajectories, together with the target RCI set and the optimized first tube sections along one representative trajectory. The tube sections contain the corresponding closed-loop state and contract toward \(X(y^\circ)\).

\begin{figure}
\centering
\includegraphics[width=\linewidth]{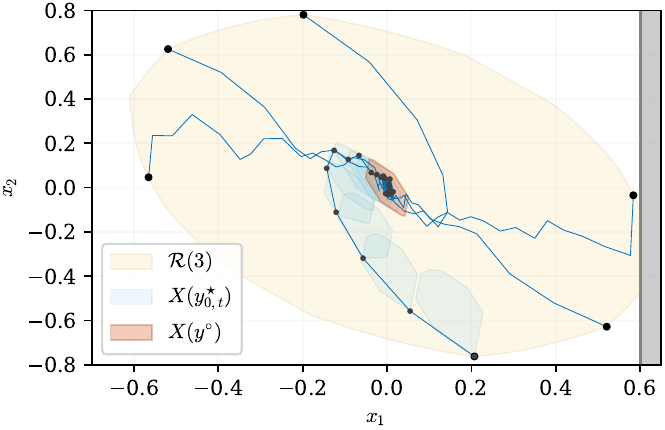}
\caption{Closed-loop trajectories, feasible-region approximation \(\mathcal R(3)\), target RCI set \(X(y^\circ)\), and optimized first tube sections \(X(y_{0,t}^\star)\) along one representative trajectory.}
\label{fig:STATE_SPACE}
\end{figure}

Figure~\ref{fig:LYAP} reports the optimal MPC cost along the same trajectories. Its decrease (up to the numerical tolerance) is consistent with the convergence result in Theorem~\ref{thm::stability}.

\subsection{Computational Aspects}
\label{sec:computational_aspects}

The experiments were coded in Python\footnote{Code is available at \url{https://github.com/fil-bad/NonlinearCCTMPC}.} and performed on an Intel Core i5-8350U CPU. The offline template transformation is modeled with CasADi and solved through IPOPT~\cite{ipopt_waechter2006implementation}. The target RCI, support-function, and online CC-TMPC problems are implemented in CVXPY~\cite{cvxpy} following disciplined convex programming (DCP) and disciplined parametrized programming (DPP) rules, and solved with Clarabel~\cite{goulart2026clarabel}. The barycentric QP in~\eqref{eq:lambda_selection} is solved with DAQP~\cite{arnstrom2022dual}.

Problem~\eqref{eq:CCTMPC_new} contains $(N+1)(\fe+n_u\ve)$ decision variables. Each constraint set $\widetilde {\mathbb{S}}$ comprises \(\ee+\ve\left(n_{\mathbb X}+n_{\mathbb U}+\pe\fe\right)\) scalar inequalities, where $n_{\mathbb X}$ and $n_{\mathbb U}$ are the numbers of half-spaces defining the state and input sets. These counts coincide with those of linear CC-TMPC for a polytopic LDI, with \(\pe\) corresponding to the number of vertex models; the propagation constraints are, however, general convex rather than affine. The barycentric QP has \(\ve\) variables, \(n_x+1\) equality constraints, and \(\ve\) nonnegativity constraints. DAQP solves this QP in a few microseconds, while its complete evaluation through the CVXPY interface requires \(\approx2\,\mathrm{ms}\). Before recording the online timings, a one-time warm-up solve at the target RCI compiles the DPP representation and initializes the conic solver. This requires approximately \(4.46\,\mathrm{s}\) for the dodecagonal template and is excluded from the reported timings. 

To assess the trade-off between feasible-region size and online complexity, the experiment is repeated for \(\fe=\ve\in\{6,8,10,12,16,20\}\), with \(N=3\) and 100 support directions. Figure~\ref{fig:LYAP_SCALABILITY}\subref{fig:SCALABILITY} reports the median end-to-end computation time and the area of \(\mathcal R(3)\). The observed 95th-percentile solution time increases from \(29.3\,\mathrm{ms}\) for \(\ve=6\) to \(226.4\,\mathrm{ms}\) for \(\ve=20\). Increasing the number of vertices enlarges the feasible region, but yields diminishing geometric returns and a superlinear increase in computation time. For this example, templates with 10 or 12 vertices provide the most favorable trade-off.

\begin{figure}
\centering
\begin{subfigure}[t]{0.49\linewidth}
    \centering
    \includegraphics[width=\linewidth]{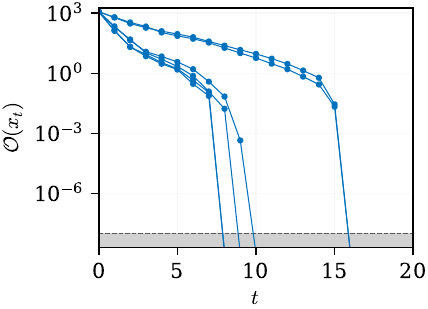}
    \caption{}
    \label{fig:LYAP}
\end{subfigure}
\hfill
\begin{subfigure}[t]{0.49\linewidth}
    \centering
    \includegraphics[width=\linewidth]{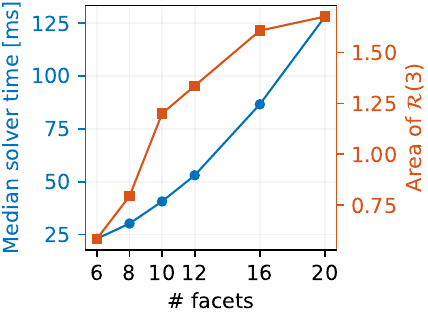}
    \caption{}
    \label{fig:SCALABILITY}
\end{subfigure}
\caption{Closed-loop convergence and computational scalability. Panel~(\subref{fig:LYAP}) shows the optimal MPC cost, while panel~(\subref{fig:SCALABILITY}) compares the median online computation time with the area of the feasible-region approximation.}
\label{fig:LYAP_SCALABILITY}
\end{figure}

\section{Conclusions}
\label{sec:conclusions}

This paper extended configuration-constrained tube MPC to uncertain nonlinear systems admitting a DC decomposition. Convex directional bounds of the dynamics and the affine vertex representation of configuration-constrained polytopes yield a finite-dimensional robust propagation certificate. The resulting controller optimizes a flexible polytopic tube and its vertex inputs within a single convex program, with recursive feasibility and convergence to a target robust control invariant set guaranteed by an implicit terminal condition.

Future work will consider nonlinear reference tracking through artificial reference variables and robust tracking tubes~\cite{8270613,polver2025robust}, and a more systematic analysis of restrictions on the tube geometry and vertex control law to balance conservatism and online complexity~\cite{badalamenti2025_efficientCCTMPC}.


\appendix
\section{Sign-Preserving Template Transformation}
\label{app:template_dc}

The feasibility and conservatism of the one-step certificate depend on the selected template normals. In the DC setting, their signs also determine the branches of the directional upper bound in Lemma~\ref{lem:c-bound}. We therefore introduce an offline transformation that adapts a seed template to the system dynamics while preserving its sign pattern. The transformation may be optimized according to different geometric criteria; in the numerical example, it is chosen to maximize the volume of the resulting RCI set.

Fix $\bar F$ and $\bar y$ such that
\[
\bar X
=
\{x\in\mathbb R^{n_x}\mid \bar F x\le\bar y\}
=
\operatorname{conv}\{\bar{x}_j\}_{j\in\iv{1,\ve}}
\]
expresses a full-dimensional seed polytope. For an invertible matrix
\(T\in\mathbb R^{n_x\times n_x}\), consider the transformed polytope
\begin{equation}\label{eq:transf_X}
X = \{x\in\mathbb R^{n_x}\mid \bar FTx\le\bar y\} = T^{-1}\bar X.
\end{equation}
Thus, any vertex \(x_j\), $j\in\iv{1,\ve}$, satisfies \(Tx_j=\bar{x}_j\).
For each facet \(l\in\iv{1,\fe}\), define the sign partition
\[
\mathcal I_l^+
=
\mathcal I^+(\bar F_l^\top),
\qquad
\mathcal I_l^-
=
\mathcal I^-(\bar F_l^\top),
\]
according to~\eqref{eq:sign_set_defn}. The transformation is restricted to the corresponding sign by imposing
\begin{equation}
\label{eq:app_sign_chamber}
(\bar{F}_l T)_i\ge 0,\quad i\in\mathcal I_l^+,\qquad (\bar{F}_l T)_i<0,\quad i\in\mathcal I_l^-,
\end{equation}
for all \(l\in\iv{1,\fe}\), with strict inequalities enforced by a small numerical margin \(\varepsilon_s>0\) in the implementation. Consequently, the transformed facet directions use the same sign partitions of the DC upper bound as the seed normals.

For a given transformation \(T\), define
\[
\widetilde f_{\bar F T}(x,u,\theta)
\coloneqq
\begin{bmatrix}
\widetilde f((x,u,\theta);\bar z,(\bar F_1T)^\top)\\
\vdots\\
\widetilde f((x,u,\theta);\bar z,(\bar F_{\fe}T)^\top)
\end{bmatrix}.
\]
Let \(\theta_k\), $k\in\iv{1,\pe}$, and
$w_q$, $q\in\iv{1,\we}$, denote the vertices of \(\Theta\) and \(\mathbb W\), respectively. The support of the additive disturbance along the transformed facet directions is upper-bounded by \(d\in\mathbb R^{\fe}\) satisfying
\[
d_l\ge \bar F_l T w_q,
\qquad
(l,q)\in\iv{1,\fe}\times\iv{1,\we}.
\]

A RCI realization of the transformed template is obtained from the offline NLP
\begin{equation}
\label{eq:app_offline_template_dc}
\scalemath{0.9}{
\begin{aligned}
\min_{T,\{x_j\},\mathbf u,d}\quad
& \Phi(T)\\
\mathrm{s.t.}\quad
& Tx_j=\bar x_j,
&& j\in\iv{1,\ve},\\
& (\bar F_l T)_i\ge0,
&& i\in\mathcal I_l^+,\quad l\in\iv{1,\fe},\\
& (\bar F_l T)_i\le-\varepsilon_s,
&& i\in\mathcal I_l^-,\quad l\in\iv{1,\fe},\\
& d_l\ge \bar F_l T w_q,
&& (l,q)\in\iv{1,\fe} \times\iv{1,\we},\\
& \widetilde f_{\bar FT}
  (x_j,U_j\mathbf u,\theta_k)+d
  \le\bar y,
&& (j,k)\in\iv{1,\ve}\times\iv{1,\pe},\\
& x_j\in\mathbb X,\qquad
  U_j\mathbf u\in\mathbb U,
&& j\in\iv{1,\ve}.
\end{aligned}
}
\end{equation}
Any feasible solution of~\eqref{eq:app_offline_template_dc} certifies that the set \(X\) in \eqref{eq:transf_X} is robust control invariant. The objective \(\Phi\) can be selected according to the desired geometric properties of the resulting set; for instance, volume maximization can be cast as
\[
\Phi(T)\coloneqq \det T, \qquad\text{s.t.}\quad\det T\ge\varepsilon_{\det}>0.
\]
Since the offset vector \(\bar y\) is fixed and \(\det T>0\),
\[
\operatorname{vol}(X)
=
\frac{\operatorname{vol}(\bar X)}{\det T}.
\]
To avoid explicit determinant evaluation, one may instead use the facet-distance surrogate
\[
\Phi(T)
=
\left\|
\operatorname{diag}(\bar y)^{-1}\bar FT
\right\|_{\mathrm F}^2, \qquad\text{s.t.}\quad\bar y>0.
\]

\begin{rem}
Invertible linear transformations preserve the face lattice of a polytope (cf.~\cite[\S 2.6]{Ziegler2008-hp}). Hence, if
\( \left\langle
\bar F,\bar E,\bar{\boldsymbol{W}}\right\rangle \) is a configuration triple for \(\bar X\), then the transformed family admits the configuration triple
\[
F=\bar FT,
\qquad
E=\bar E,
\qquad
W_j=T^{-1}\bar W_j,
\quad
j\in\iv{1,\ve}.
\]
\end{rem}

\bibliographystyle{plain}        
\bibliography{new_refs}           

\begin{thebibliography}{10}

\bibitem{althoff2013reachability}
M.~Althoff.
\newblock Reachability analysis of nonlinear systems using conservative polynomialization and non-convex sets.
\newblock In {\em Proceedings of the 16th International Conference on Hybrid Systems: Computation and Control}, pages 173--182, 2013.

\bibitem{althoff2008reachability}
M.~Althoff, O.~Stursberg, and M.~Buss.
\newblock Reachability analysis of nonlinear systems with uncertain parameters using conservative linearization.
\newblock In {\em Proceedings of the 47th IEEE Conference on Decision and Control}, pages 4042--4048, 2008.

\bibitem{arnstrom2022dual}
D.~Arnstr{\"o}m, A.~Bemporad, and D.~Axehill.
\newblock A dual active-set solver for embedded quadratic programming using recursive {LDL}$^{T}$ updates.
\newblock {\em IEEE Transactions on Automatic Control}, 67(8):4362--4369, 2022.

\bibitem{badalamenti2025_efficientCCTMPC}
F.~Badalamenti, S.~K. Mulagaleti, M.~E. Villanueva, B.~Houska, and A.~Bemporad.
\newblock Efficient configuration-constrained tube {MPC} via variables restriction and template selection.
\newblock In {\em Proceedings of the 64th IEEE Conference on Decision and Control}, pages 1783--1789, 2025.

\bibitem{fb}
F.~Blanchini and S.~Miani.
\newblock {\em Set-Theoretic Methods in Control}.
\newblock Birkh{\"a}user, Boston, MA, 2015.

\bibitem{bravo2003robust}
J.~M. Bravo, T.~Alamo, D.~Lim{\'o}n, and E.~F. Camacho.
\newblock Robust mpc of constrained discrete-time nonlinear systems based on zonotopes.
\newblock In {\em 2003 European Control Conference (ECC)}, pages 2035--2040. IEEE, 2003.

\bibitem{bravo2005computation}
J.~M. Bravo, D.~Lim{\'o}n, T.~Alamo, and E.~F. Camacho.
\newblock On the computation of invariant sets for constrained nonlinear systems: An interval arithmetic approach.
\newblock {\em Automatica}, 41(9):1583--1589, 2005.

\bibitem{cvxpy}
S.~Diamond and S.~Boyd.
\newblock {CVXPY}: A {P}ython-embedded modeling language for convex optimization.
\newblock {\em Journal of Machine Learning Research}, 17(83):1--5, 2016.

\bibitem{DIEHL20081279}
M.~Diehl, J.~Gerhard, W.~Marquardt, and M.~M{\"o}nnigmann.
\newblock Numerical solution approaches for robust nonlinear optimal control problems.
\newblock {\em Computers \& Chemical Engineering}, 32(6):1279--1292, 2008.

\bibitem{doff2026computationally}
M.~Doff-Sotta, Z.~A-Rahman, and M.~Cannon.
\newblock Computationally tractable robust nonlinear model predictive control using {DC} programming.
\newblock {\em arXiv preprint arXiv:2602.01164}, 2026.

\bibitem{NLTMPC_9993390}
M.~Doff-Sotta and M.~Cannon.
\newblock Difference of convex functions in robust tube nonlinear {MPC}.
\newblock In {\em Proceedings of the 61st IEEE Conference on Decision and Control}, pages 3044--3050, 2022.

\bibitem{FIACCHINI2010}
M.~Fiacchini, T.~Alamo, and E.~F. Camacho.
\newblock On the computation of convex robust control invariant sets for nonlinear systems.
\newblock {\em Automatica}, 46(8):1334--1338, 2010.

\bibitem{FIACCHINI2012819}
M.~Fiacchini, T.~Alamo, and E.~F. Camacho.
\newblock Invariant sets computation for convex difference inclusion systems.
\newblock {\em Systems \& Control Letters}, 61(8):819--826, 2012.

\bibitem{goulart2026clarabel}
P.~J. Goulart and Y.~Chen.
\newblock Clarabel: An interior-point solver for conic programs with quadratic objectives.
\newblock {\em Mathematical Programming Computation}, 2026.

\bibitem{horst1999dc}
R.~Horst and N.~V. Thoai.
\newblock {DC} programming: Overview.
\newblock {\em Journal of Optimization Theory and Applications}, 103(1):1--43, 1999.

\bibitem{Kohler2021}
J.~K{\"o}hler, R.~Soloperto, M.~A. M{\"u}ller, and F.~Allg{\"o}wer.
\newblock A computationally efficient robust model predictive control framework for uncertain nonlinear systems.
\newblock {\em IEEE Transactions on Automatic Control}, 66(2):794--801, 2021.

\bibitem{Kolmanovsky1998}
I.~Kolmanovsky and E.~G. Gilbert.
\newblock Theory and computation of disturbance invariant sets for discrete-time linear systems.
\newblock {\em Mathematical Problems in Engineering}, 4(4):317--367, 1998.

\bibitem{Langson2004}
W.~Langson, I.~Chryssochoos, S.~V. Rakovi{\'c}, and D.~Q. Mayne.
\newblock Robust model predictive control using tubes.
\newblock {\em Automatica}, 40(1):125--133, 2004.

\bibitem{8270613}
D.~Limon, A.~Ferramosca, I.~Alvarado, and T.~Alamo.
\newblock Nonlinear {MPC} for tracking piece-wise constant reference signals.
\newblock {\em IEEE Transactions on Automatic Control}, 63(11):3735--3750, 2018.

\bibitem{Lishkova2025}
Y.~Lishkova and M.~Cannon.
\newblock A successive convexification approach for robust receding horizon control.
\newblock {\em IEEE Transactions on Automatic Control}, 70(10):6436--6448, 2025.

\bibitem{Lopez2019}
B.~T. Lopez, J.-J.~E. Slotine, and J.~P. How.
\newblock Dynamic tube {MPC} for nonlinear systems.
\newblock In {\em Proceedings of the American Control Conference}, pages 1655--1662, 2019.

\bibitem{MAYNE_RakovicRigidTMPC2005219}
D.~Q. Mayne, M.~M. Seron, and S.~V. Rakovi{\'c}.
\newblock Robust model predictive control of constrained linear systems with bounded disturbances.
\newblock {\em Automatica}, 41(2):219--224, 2005.

\bibitem{polver2025robust}
M.~Polver, D.~Limon, F.~Previdi, and A.~Ferramosca.
\newblock Robust tracking mpc for perturbed nonlinear systems.
\newblock {\em IEEE Transactions on Automatic Control}, 71(1):214--229, 2025.

\bibitem{Rakovic2013_LDI_HTMPC}
S.~V. Rakovi{\'c} and Q.~Cheng.
\newblock Homothetic tube {MPC} for constrained linear difference inclusions.
\newblock In {\em Proceedings of the 25th Chinese Control and Decision Conference}, pages 754--761, 2013.

\bibitem{Sasfi2023}
A.~Sasfi, M.~N. Zeilinger, and J.~K{\"o}hler.
\newblock Robust adaptive {MPC} using control contraction metrics.
\newblock {\em Automatica}, 155:111169, 2023.

\bibitem{Villanueva2020}
M.~E. Villanueva, E.~De~Lazzari, M.~A. M{\"u}ller, and B.~Houska.
\newblock A set-theoretic generalization of dissipativity with applications in tube {MPC}.
\newblock {\em Automatica}, 122:109179, 2020.

\bibitem{CCTMPC_VILLANUEVA2024111543}
M.~E. Villanueva, M.~A. M{\"u}ller, and B.~Houska.
\newblock Configuration-constrained tube {MPC}.
\newblock {\em Automatica}, 163:111543, 2024.

\bibitem{ipopt_waechter2006implementation}
A.~W{\"a}chter and L.~T. Biegler.
\newblock On the implementation of an interior-point filter line-search algorithm for large-scale nonlinear programming.
\newblock {\em Mathematical Programming}, 106(1):25--57, 2006.

\bibitem{Ziegler2008-hp}
G.~M. Ziegler.
\newblock {\em Lectures on Polytopes}, volume 152 of {\em Graduate Texts in Mathematics}.
\newblock Springer, New York, NY, 2008.

\end{thebibliography}

\end{document}